\documentclass[12pt]{article}
\usepackage{graphicx} 
\usepackage{tikz-cd}
\usepackage[utf8]{inputenc}
\usepackage{amsmath}
\usepackage{amsfonts}
\usepackage{amssymb}
\usepackage{amsthm}
\usepackage{booktabs}
\usepackage{latexsym}
\usepackage{fullpage}
\usepackage{ytableau}
\usepackage{tikz}
\usepackage[
    backend=biber,
    style=numeric,
  ]{biblatex}
\usepackage{float}
\usepackage{mathdots}
\usepackage[multiple]{footmisc}

\newtheorem{thm}{Theorem}
\newtheorem{prop}{Proposition}

\newtheorem{defin}{Definition}
\newtheorem{rmk}{Remark}

\newtheorem{cor}{Corollary}
\newtheorem{lem}{Lemma}

\title{A new kind of automorphic form and a proof of the essential transformation laws}
    
\author{Michael Andrew Henry\footnote{The author can be reached uninterrupted at locrian.gates@gmail.com. This research has been supported by the Austrian Science Fund (FWF), grant number W1230. } [TU Graz]}

\date{\today}

\begin{document}

\maketitle

\begin{abstract}
    We utilize the structure of quasiautomorphic forms over an arbitrary Hecke triangle group to define a new vector analogue of an automorphic form. We supply a proof of the functional equations that hold for these functions modulo the group generators.    
\end{abstract}


\section{Introduction}
In Dijkgraaf's influential work \cite{dijkgraaf1995}, we find a quasimodular form making an appearance as a generating function with applications in mirror symmetry. The study \cite{kanekoZagier1995} cemented these functions as robust enumerative tools with an algebraic structure similar to, but ultimately more complicated than, modular forms. For us, modular forms \cite{zagier2008} are just automorphic forms with respect to the modular group, dispelling any naming confusion. 

In this essay we develop further the theory of quasiautomorphic forms. In particular, we utilize an arbitrary quasiautomorphic form $ U(z) $ with respect to some Hecke triangle group $ \mathfrak{t}_\mu $ to define an analogue $ \mathbf{F}_U (z) $ of an automorphic form \emph{qua} $ \mathfrak{t}_\mu $. The function $ \mathbf{F}_U $ is the ``new automorphic form'' announced in the title. The seemingly paradoxical construction works because, beginning with the holomorphic function $ U $ for $ z \in \mathbb{H} $, the resulting object $ \mathbf{F}_U $ is now a vector function, something we call a Hecke vector-form. After defining a Hecke vector-form, we prove the transformation formulas for $ \mathbf{F}_U $ modulo the generators of the associated Hecke triangle group.  

\section{Quasiautomorphic forms on Hecke triangle groups}
Detailed background of triangle groups and automorphic forms for this section can be found in \cite{doranEtAl2013} or with an emphasis similar to this paper in \cite{henry2025}; this note is a simplified and condensed exposition of the most fundamental ideas of the latter study. 

Let us begin with a quasiautomorphic form $ U_{w,r}(z) $ ($ z\in\mathbb{H} $) of weight $ w $ and depth $ r $ over some arbitrary Hecke triangle group, $ \mathfrak{t}_\mu = (2, \mu, \infty) $ for when $ 3 \leq \mu \in \mathbb{Z} $ holds (this triplet notation is  fully explained in \cite{hille1962} and \cite{nehari1952}). We recall that if $ 0 \leq r \leq w/2 $, then
    \begin{equation}\label{eq:standarForm}
        U_{\mathfrak{t}_\mu, w,r}(z) :=  \sum_{k=0}^{r} H_{w-2k}(z)\, E^k_2 (z) 
    \end{equation}
where $ H_{w-2k} \in \mathcal{A}_{w-2k}(\mathfrak{t}_\mu) $ i.e. $ H_{w-2k}(z) $ is an element of the vector space of automorphic forms of weight $  w -2k $ with respect to $ \mathfrak{t}_\mu $.  For orientation, the Hecke triangle group $ \mathfrak{t}_3 = (2,3,\infty) $ is just the classical modular group and thus our study entails results about the classical quasimodular forms. Regarding the notation $ U_{\mathfrak{t}_\mu, w, r} $, we will oscillate between writing $ U $, $ U_{ w, r} $, $U_{\mathfrak{t}_\mu, w, r} $, depending on the desired emphasis.

We also introduce and use throughout the notation
    \begin{eqnarray*}
        \{ r,\ell \}_m :=  \frac{(m+\ell)!(r-\ell)!}{m!\,r!} 
    \end{eqnarray*}
to avoid clutter (for us the condition $ r \geq \ell \geq 0 $ will always be satisfied, $ \ell, r \in \mathbb{N}\cup\{0\} $ holding). With respect to this notation, observe that the equality 
    \begin{eqnarray}\label{eq:coeffEquiv}
        \binom{m+p}{p}\{r,\ell \}_{m+p} = \binom{r-\ell}{m}\{r,m+\ell \}_p
    \end{eqnarray}
holds. 

To establish our theory, beginning with $ U_{\mathfrak{t}_\mu, w,r} $, we will consider a set of functions dependent on $ U_{\mathfrak{t}_\mu, w,r} $ as follows:
    \begin{eqnarray*}
        f_0(z) &=& g_0 \\
        f_1(z) &=& g_0z + g_1  \\
        f_2(z) &=& g_0z^2 + 2g_1z + g_2  \\
        f_3(z) &=& g_0z^3 + 3g_1z^2 + 3g_2z + g_3  \\
        & \vdots & 
    \end{eqnarray*}
or elements defined generally by 
    \begin{eqnarray}\label{eq:defFAux}
        f_n(z) &:=& \sum_{k=0}^n \binom{n}{k}g_{k}z^{n-k},
    \end{eqnarray}
($ 0 \leq n \leq r\leq w/2$ satisfied) where $ g_{\ell} $ is the function
    \begin{eqnarray}\label{ex:defGAux}
        g_{\ell}(z) := C^\ell  \sum_{m=0}^{r-\ell}\{ r,\ell \}_m B_{U, \ell + m}(z)\, E_{2}^{m}(z),
    \end{eqnarray}
$ C_{\mathfrak{t}_\mu} = C  $ is the unique structure constant associated to the weight $ w =2 $ Eisenstein series $ E_{2}(z) $ modulo $ \mathfrak{t}_\mu $ (i.e. $ \lim_{z\rightarrow i\infty}E_{\mathfrak{t}_\mu, 2}(z) = C_{\mathfrak{t}_\mu} $
holds), and $  B_{U, k} : = H_{w-2k}  $ from the sum representation of $ U_{\mathfrak{t}_\mu, w,r} $ given at \eqref{eq:standarForm}. As we should expect from \eqref{ex:defGAux}, $ z $ will always be restricted to the upper-half complex plane. 

    \begin{rmk}
        Given $ U_{\mathfrak{t}_\mu, w,r} $, it is evident that $ g_0 $ is a weight $ w $ and depth $ r $ quasiautomorphic form modulo $ \mathfrak{t}_\mu $; similarly, each function $ g_k $ is a quasiautomorphic form of weight $ w-2k $, depth $ r-k  $ or what can be viewed simply as a truncation of $ g_0 $ such that homogeneity of the weight is preserved in each term. 
    \end{rmk}

To motivate, the finite sequence $ \mathbf{f} = (f_k)^r_{k=0} $ of elements given at \eqref{eq:defFAux} comes from defining two auxiliary sequences $ \mathbf{g} = (g_\ell)^r_{\ell =0}
 $ and $ \mathbf{z} = (z^j)^r_{j=0} $ such that the equality
    \begin{eqnarray}\label{ex:fAux}
        f_n = \sum_{k=0}^{n}\binom{n}{k}g_{k}{z^{n-k}}    
    \end{eqnarray}
holds i.e. $ f_n $ is defined as the $ n $-th binomial convolution of $ \mathbf{g} $ and $ \mathbf{z} $ where $ 0 \leq n \leq r $ holds. A convolution may be found written elsewhere as $ f_n = (\mathbf{g}\circ\mathbf{z})_n $, but we mostly avoid this notation; however, for background on this idea see \cite{grahamKnuthPatashnik1989}. One immediate implication of the set-up is we have that the so-called orthogonality relation
    \begin{eqnarray*}\label{eq:orthogonality}
        \frac{f_n}{z^n} = \sum_{i=0}^{n} \binom{n}{i}\frac{g_i}{z^{i}} \quad \Longleftrightarrow \quad \frac{g_n}{z^n}= \sum_{j=0}^{n}(-1)^{n-j} \binom{n}{j}{f_j}{z^j} 
    \end{eqnarray*}
holds.

\section{Matrix theory for Hecke vector-forms}
We will need some special matrices, some for which we use non-standard notation (the reader can find any unmentioned background in \cite{hornJohnson1985}). 
    
    \begin{rmk}
        In our matrix notation, an absence of any indicating entry $ a_i $ always means $ a_i = 0 $ i.e. a blank space is to be assumed to take a zero.
    \end{rmk}
    
We can motivate the introduction of matrix theory here by observing
    \begin{eqnarray}\label{eq:matrixConvolutionIdentity}
        \begin{pmatrix}
            f_0 \\ f_1 \\ \vdots \\ f_r
        \end{pmatrix}
            =
        \begin{pmatrix}
            \binom{0}{0} \\
            \binom{1}{0}z & \binom{1}{1} \\
            \vdots &\vdots & \ddots \\
            \binom{r}{0}z^r & \binom{r}{1}z^{r-1} & \cdots & \binom{r}{r} 
        \end{pmatrix}
        \begin{pmatrix}
            g_0 \\
            g_1 \\
            \vdots \\
            g_r
        \end{pmatrix}
    \end{eqnarray}
holds, involving some pleasingly natural and well-behaved matrices. Clearly, \eqref{eq:matrixConvolutionIdentity} is just the matrix summary of the convolution structure we saw in the previous section.

    \begin{defin}
        For the $ (r+1)\times(r+1) $ lower triangular matrix known as the \bf{generalized Pascal matrix}\normalfont, we write
            \begin{eqnarray*}
                P_r(z) := 
                    \begin{pmatrix}
                        \binom{0}{0} \\
                        \binom{1}{0}z & \binom{1}{1} \\
                        \vdots &\vdots & \ddots \\
                        \binom{r}{0}z^r & \binom{r}{1}z^{r-1} & \cdots & \binom{r}{r} 
                    \end{pmatrix}.
                \end{eqnarray*}
            \end{defin}
For more about these matrices see e.g. \cite{acetoTriguante2001} and \cite{edelmanStrang2004} and the references cited therein.  
    \begin{defin}\label{def:creationMatrix}
        Let
            \begin{eqnarray*}
                A_r(\lambda) := 
                    \begin{pmatrix}
                         \lambda  \\
                        1 & \lambda  \\
                        & \ddots & \ddots &  \\
                        & & r& \lambda
                    \end{pmatrix}.
            \end{eqnarray*}
        Then the so-called \bf creation matrix \normalfont is defined by the special instance, $  A_r := A_r(0) $. 
    \end{defin}
    
The terminology in Definition \ref{def:creationMatrix} comes from the paper \cite{acetoMalonekTomaz2015}. It is both well-known and useful that the relation $A_r^{s} = {0}$ holds for $ s \geq r + 1 $ i.e. the creation matrix $ A_r $ is a nilpotent matrix of index $ r+1 $.

Additionally, it can be checked that the characteristic polynomial of $ A_r^{\mathbf{t}}(\lambda) $ satisfies
    \begin{eqnarray*}
        \text{{char}}_{A_r(\lambda)}(X) = \sum^{r}_{k=0}\binom{r}{k}X^{r-k}\lambda^k = (X - \lambda)^r.
    \end{eqnarray*}
One interpretation of this via the Cayley-Hamiliton theorem is $ A_r^{\mathbf{t}}(\lambda) $ is similar to a Jordan matrix, $ J_r(\lambda) $, one of the most essential matrices \cite{hornJohnson1985}.

    \begin{defin}
        The $(r+1) \times (r+1) $ so-called \bf exchange matrix \normalfont we denote 
            \begin{eqnarray*}
                \iota_r := 
                    \begin{pmatrix}
                        &&&1\\
                        &&1\\
                        & \iddots \\
                        1
                    \end{pmatrix}.
            \end{eqnarray*}
    \end{defin}

The exchange matrix (as it is called in \cite{hornJohnson1985}) is often hidden behind the matrix transpose operator. For example, with some loss of generality, for a $ (r+1) \times (r+1) $ matrix $ M $, the relation $ M^\mathbf{t} = \iota_r M \iota_r  $ holds. We shall want to have on hand a left ``half-transpose'' and a right ``half-transpose,'' namely the dual operators
    \begin{eqnarray*}
        M^\mathbf{x} := \iota_r M \quad \text{and} \quad
        M^\mathbf{y} := M\iota_r.
    \end{eqnarray*}

We shall also want to use the familiar $(r+1) \times (r+1)$ diagonal matrix that we denote by
    \begin{eqnarray*}
        d_r(a_i) :=
            \begin{pmatrix}
                a_0\\
                & a_1 \\
                &&\ddots \\
                &&& a_r
            \end{pmatrix}.
    \end{eqnarray*}

For the first application of matrix theory, we utilize the matrix exponential function $ e^M := \sum_{n=0}^{\infty}\frac{M^n}{n!} $, $M$ square, to give
   \begin{lem}\label{lem:pascalExp}
        The relation
            \begin{eqnarray*}
                P_r(z) = e^{z A_r}
            \end{eqnarray*}
        holds. 
    \end{lem}
    \begin{proof}
        Exercise.
    \end{proof}

For what is to follow, let us fix the notation
    \begin{eqnarray*}
        \mathbf{F}_U(z) := 
            \begin{pmatrix}
                f_0(z) & f_1(z) & \cdots & f_r(z)
            \end{pmatrix}^{\mathbf{t}}
    \end{eqnarray*}
and
    \begin{eqnarray*}
        \mathbf{G}_U(z) := 
            \begin{pmatrix}
                g_0(z) & g_1(z) & \cdots & g_r(z)
            \end{pmatrix}^{\mathbf{t}},
    \end{eqnarray*}
where we have made the dependence of these vectors on $ U $ explicit. 

    \begin{lem}\label{lem:expForm}
        The relation 
            \begin{eqnarray*}
                \mathbf{F}_U(z) = e^{z A_r} \mathbf{G}_U(z)
            \end{eqnarray*}
        holds. 
    \end{lem}
    \begin{proof}
        Using Lemma \ref{lem:pascalExp} and \eqref{eq:matrixConvolutionIdentity} the result is immediate. 
    \end{proof}

We can also give a second useful representation of $ \mathbf{F}_U $. Again, depending on $ U $, suppose that we write (along lines suggested by $ P_r(z) $)
    \begin{eqnarray*}
        P(\mathbf{G}_U)(z) :=
            \begin{pmatrix}
                \binom{0}{0}g_0(z) \\
                \binom{1}{0} g_1(z) & \binom{1}{1}g_0(z) \\
                \vdots & \vdots & \ddots \\
                \binom{r}{0}g_r(z) & \binom{r}{1}g_{r-1}(z) & \cdots & \binom{r}{r}g_{0}(z)  
            \end{pmatrix}  
    \end{eqnarray*}
and the monomial-powers vector we write as
    \begin{equation*}
        \nu_r(z) := 
        \begin{pmatrix}
            1 & z & \cdots & z^r 
        \end{pmatrix}^\mathbf{t}.
    \end{equation*}
The following fact is a simple application of this notation.

    \begin{lem}\label{lem:VectorFormEquivalence}
        The relation
            \begin{eqnarray*}
                \mathbf{F}_{U}(z) = P(\mathbf{G}_U)(z)\,\nu_r(z)
            \end{eqnarray*}
        holds. 
    \end{lem}
    \begin{proof}
        This is clear, considering the meaning of the notation  $ P(\mathbf{G}_U)(z) $ alongside \eqref{eq:matrixConvolutionIdentity}.
    \end{proof}

In a culmination of the section we introduce a
    \begin{defin}\label{def:HeckeVectorForm}
        We call the function $ \mathbf{F}_U\left(z \right) $ defined on the upper-half complex plane the \bf{Hecke vector-form} \normalfont of $ U $; the function $ \mathbf{G}_U(z) $ of Lemma \ref{lem:expForm} is called the \bf{hauptbuch} \normalfont of $U$; the function $ P(\mathbf{G}_U)(z) $ of Lemma \ref{lem:VectorFormEquivalence} is called the \bf{transfer matrix }\normalfont of $ U $.
    \end{defin}

    \begin{rmk}
        Definition \ref{def:HeckeVectorForm} comes originally from \cite{henry2025}, albeit there in a slightly more complicated form.
    \end{rmk}

\section{Transformation equations for the generators}
The precedent for this section comes from \cite{grabner2020}, \emph{viz} Proposition 3.2, but we have interpreted the result differently.

We recall the generators of the Hecke triangle group, $ \mathfrak{t}_\mu $. If $ \varpi_\mu := 2 \cos \frac{\pi}{\mu} $, then for $ z \in \mathbb{H} $
    \begin{eqnarray*}
        Tz :=  z + \varpi_\mu \quad \text{and} \quad 
        Sz  :=  -\frac{1}{z}
    \end{eqnarray*}
are the two independent generators of $ \mathfrak{t}_\mu $.

    \begin{rmk}
        We see that $ Tz = z + \varpi_\mu $ depends on $\mathfrak{t}_\mu $; however, evidently, $ Sz $ remains fixed. 
    \end{rmk}

The next goal will be to determine the behavior of Hecke vector forms under these two maps essential to $ \mathfrak{t}_\mu $. 
    \begin{lem}\label{lem:hauptbuchUnderT}
        Let $ U_{\mathfrak{t}_\mu, w, r} $ be some quasiautomorphic form. If $\mathbf{G}_U(z) $ is the hauptbuch of $ U_{\mathfrak{t}_\mu, w, r} $, then 
            \begin{eqnarray*}
                \mathbf{G}_U(Tz) = \mathbf{G}_U(z)
            \end{eqnarray*}
        holds.
    \end{lem}
    \begin{proof}
        Recall that $ U_{\mathfrak{t}_\mu, w, r} $ is periodic under $ Tz $ and thus $ g_\ell $ is periodic under $ Tz $.
    \end{proof}

    \begin{cor}
        The relation
            \begin{eqnarray*}
                P(\mathbf{G}_U)(Tz) = P(\mathbf{G}_U)(z) 
            \end{eqnarray*}
        holds. 
    \end{cor}
    \begin{proof}
        This is clear from Lemma \ref{lem:hauptbuchUnderT}.
    \end{proof}

    \begin{thm}[$ \mathbf{F}_U $ under $ Tz $]\label{thm:vectorFormUnderT}
        Let $ U_{\mathfrak{t}_\mu, w,r} $ be as before. The relation 
            \begin{eqnarray*}
                \mathbf{F}_U(Tz) = e^{\varpi_\mu A_r} \mathbf{F}_U(z)
            \end{eqnarray*}
        holds. 
    \end{thm}
    \begin{proof}
        This follows from Lemma \ref{lem:expForm}, Lemma \ref{lem:hauptbuchUnderT}, and the fact that $ e^{(a+b)M} = e^{aM}e^{bM} $ holds for scalars $ a $ and $b$. 
    \end{proof}

The case of $ \mathbf{F}_U $ under $ Sz $ is a little more involved. We will build up to the statement.  It will become clear at this point why we have used the convolution structure. 
    \begin{prop}\label{prop:gUnderS}
        The relation 
            \begin{eqnarray*}
                \frac{g_\ell(Sz)}{z^{w-r-\ell}} = \sum_{m=0}^{r-\ell}\binom{r-\ell}{m}g_{\ell+m}z^{r-\ell - m}
            \end{eqnarray*}
        holds.
    \end{prop}
    \begin{proof}
        Starting with $ g_\ell $ as defined at \eqref{ex:defGAux} and recalling the binomial theorem, the string of relations
            \begin{eqnarray*}
                g_\ell(Sz) &=& C^\ell \sum_{m=0}^{r-\ell} \{r,\ell\}_m B_{U, \ell + m}(Sz)\,E_2^m(Sz) \\
                &=& C^\ell \sum_{m=0}^{r-\ell} \{r,\ell\}_m z^{w-2(\ell+m)}B_{U, \ell + m}(z)\,(z^2E_2(z) + Cz)^m \\
                &=& C^\ell \sum_{m=0}^{r-\ell} \{r,\ell\}_m z^{w-2\ell}B_{U, \ell + m}(z)\,\left(E_2(z) + \frac{C}{z} \right)^m \\
                &=& C^\ell \sum_{m=0}^{r-\ell} \{r,\ell\}_m z^{w-2\ell}B_{U, \ell + m}(z)\,\left(\sum^m_{k=0}\binom{m}{k}E_2^{m-k}(z)\left( \frac{C}{z}\right)^{k} \right)
            \end{eqnarray*}
        hold. Apply properties of sums to get something more orderly, namely
            \begin{eqnarray*}
                C^\ell \sum_{m=0}^{r-\ell} \{r,\ell\}_m z^{w-2\ell}B_{U, \ell + m}(z)\,\left(\sum^m_{k=0}\binom{m}{k}E_2^{m-k}(z)\left( \frac{C}{z}\right)^{k} \right) &=& \\ 
                C^\ell \sum_{m=0}^{r-\ell}\sum^m_{k=0} \binom{m}{k}\{r,\ell\}_m z^{w-2\ell}B_{U, \ell + m}(z)E_2^{m-k}(z) \left( \frac{C}{z}\right)^{k} &=& \\
                z^{w-2\ell} {C^\ell}\sum_{m=0}^{r-\ell}\sum^m_{k=0} \binom{m}{k}\{r,\ell\}_m B_{U, \ell + m}(z)E_2^{m-k}(z) \left( \frac{C}{z}\right)^{k} 
            \end{eqnarray*}
        hold. Apply the law 
            \begin{eqnarray*}
                \sum_{m=0}^{r-\ell} \sum_{k=0}^{m}a_{k,m} = \sum_{k=0}^{r-\ell} \sum_{m=0}^{r-\ell -k}a_{k, m+k}
            \end{eqnarray*}
        to get (after dropping arguments from notation of $ B_{U, n}(z) $ and $ E_2(z) $ and applying $ p \mapsto k $)
            \begin{eqnarray}\label{ex:penultimate}
                z^{w-2\ell}C^\ell \sum_{p=0}^{r-\ell}\sum^{r-\ell-p}_{m=0} \binom{m+p}{p}\{r,\ell\}_{m+p} B_{U, \ell + m+p}E_2^{m}\left( \frac{C}{z}\right)^{p}
            \end{eqnarray}
        is satisfied. It is desired that we ultimately have 
            \begin{eqnarray*}
                z^{w-\ell}\sum_{p=0}^{r-\ell}\binom{r-\ell}{p} g_{\ell + p}z^{-\ell - p}
            \end{eqnarray*}
        where $ g_{\ell + p} := C^{\ell + p}\sum^{r-\ell-p}_{m=0}\{r,\ell+m\}_p B_{U,\ell+ p+m}E_2^{p} $. Using \eqref{ex:penultimate} and what we saw about properties of $ \{r,\ell\}_m $ at \eqref{eq:coeffEquiv}, we can write
            \begin{eqnarray*}
                g_{\ell}(Sz) &=& z^{w-\ell} \sum_{p=0}^{r-\ell}\sum^{r-\ell-p}_{m=0} \binom{r-\ell}{m}\{r,\ell+m\}_{p}\, B_{U, \ell +p+m}E_2^{m}\left( \frac{C}{z}\right)^{\ell+p} \\
                &=& z^{w-\ell} \sum_{p=0}^{r-\ell}\binom{r-\ell}{m}\left( {C}^{\ell+p}\sum^{r-\ell-p}_{m=0}  \{r,\ell+m\}_{p}\, B_{U, \ell +p+m}E_2^{p}\right){z}^{-\ell-p} \\
                &=& z^{w-\ell}\sum_{p=0}^{r-\ell}\binom{r-\ell}{p} g_{\ell + p}z^{-\ell - p}
            \end{eqnarray*}
        holds, where in the penultimate line, due to the invariance, we have applied $ E_2^p \mapsto E_2^m $. Therefore, we have shown that
            \begin{eqnarray*}
                \frac{g_\ell(Sz)}{z^{w-r-\ell}} = \sum^{r-\ell}_{p=0}\binom{r-\ell}{p}g_{\ell + p}z^{r-\ell-p}
            \end{eqnarray*}
        holds, which was our goal.
    \end{proof}
Proposition \ref{prop:gUnderS} is the principal tool for this article and so we have given a pedantic proof. One important implication is found in 
    \begin{cor}\label{cor:convolvUnderS}
        The relation 
            \begin{eqnarray*}
                \frac{1}{z^{w-r}}\sum_{\ell=0}^{n}\binom{n}{\ell}{g_{\ell}(Sz)}\,{(Sz)^{n-\ell}} = (-1)^n\sum_{m=0}^{r-n}g_m(z)\,z^{r-n-m}
            \end{eqnarray*}
        holds. 
    \end{cor}
    \begin{proof}
        Observe that
            \begin{eqnarray*}
               \frac{1}{z^{w-r}} \sum_{\ell=0}^{n}\binom{n}{\ell}{g_{\ell}(Sz)}\,{(Sz)^{n-\ell}} &=& \frac{(Sz)^n}{z^{w-r}} \sum_{\ell=0}^{n}\binom{n}{\ell}\frac{g_\ell(Sz)}{(Sz)^\ell} \\
               &=& \frac{1}{z^{w-r}}\frac{(-1)^n}{z^n}\sum_{\ell = 0}^{n}(-1)^\ell \binom{n}{\ell}{g_\ell(Sz)}{z^\ell}
            \end{eqnarray*}
        are satisfied. With Proposition \ref{prop:gUnderS} we see
            \begin{eqnarray*}
                {g_\ell(Sz)}\,z^\ell = z^{w}\sum_{m=0}^{r-\ell}\binom{r-\ell}{m}\frac{g_{\ell + m}(z)}{z^{\ell + m}}
            \end{eqnarray*}
        holds and thus 
            \begin{eqnarray*}
               \frac{1}{z^{w-r}}\frac{(-1)^n}{z^n}\sum_{\ell = 0}^{n}(-1)^\ell \binom{n}{\ell}{g_\ell(Sz)}{z^\ell} &=& \frac{(-1)^n}{z^{n-r}}\sum_{\ell = 0}^{n}(-1)^\ell \binom{n}{\ell}\left(\sum_{m=0}^{r-\ell}\binom{r-\ell}{m}\frac{g_{\ell + m}(z)}{z^{\ell + m}} \right) \\ &=& \frac{(-1)^n}{z^{n-r}}\sum_{\ell = 0}^{n}\sum_{m=0}^{r-\ell}(-1)^\ell \binom{n}{\ell}\binom{r-\ell}{r-(\ell + m)} \,\frac{g_{\ell + m}(z)}{z^{\ell + m}} 
            \end{eqnarray*}
        is correct. Applying the natural choice of sum law we get
            \begin{eqnarray*}
                \frac{(-1)^n}{z^{n-r}}\sum_{\ell = 0}^{n}\sum_{m=0}^{r-\ell}(-1)^\ell \binom{n}{\ell}\binom{r-\ell}{r-\ell - m} \,\frac{g_{\ell + m}}{z^{\ell + m}} &=& \frac{(-1)^n}{z^{n-r}}\sum_{m = 0}^{r}\left(\sum_{\ell=0}^{n}(-1)^\ell \binom{n}{\ell}\binom{r-\ell}{r-m}\right) \frac{g_{m}}{z^{m}} \\
                &=& {(-1)^n}\sum_{m = 0}^{r-n}\binom{r-n}{m} {g_{m}(z)}\,{z^{r-n- m}}
            \end{eqnarray*}
        holds, using Vandermonde in the penultimate line to get rid of a summation sign. 
    \end{proof}

Corollary \ref{cor:convolvUnderS} is the analogue of the transformation behavior of an automorphic form under $ Sz $, but for our quasiautomorphic form $ g_\ell $. That is, were we to have written $ f_k(z) = (\mathbf{g}\circ\mathbf{z})_k(z) $ \emph{a la} Section 2, then we see  
    \begin{eqnarray*}
        \frac{(\mathbf{g}\circ\mathbf{z})_k(Sz)}{z^{w-r}} = (-1)^k (\mathbf{g}\circ\mathbf{z})_{r-k}(z)
    \end{eqnarray*}
holds. This fact is used in proof of the transformation behavior of $ \mathbf{F}_U (Sz) $.

    \begin{thm}[$ \mathbf{F}_U $ under $ Sz $]\label{thm:vectorFormUnderS}
        The relation
            \begin{eqnarray}\label{eq:FUnderS}
                \frac{\mathbf{F}_U(Sz)}{z^{w-r}} = d_r^{\mathbf{y}} \mathbf{F}_{U}(z) 
            \end{eqnarray}
        holds if $ d_r = d_r(a_i)$ for when $ a_i = (-1)^i $ is satsified.
    \end{thm}
    \begin{proof}
       For this, we use Lemma \ref{lem:VectorFormEquivalence}. That is, we will consider 
            \begin{eqnarray*}
                \mathbf{F}_U(z) = P(\mathbf{G}_U)(z)\,\nu_r(z) 
            \end{eqnarray*}
        under the map $ Sz $. Unpacking the identity gives
            \begin{eqnarray*}
                P(\mathbf{G}_U)(Sz)\,\nu_r(Sz) =
                   \begin{pmatrix}
                        \binom{0}{0}g_0(Sz) \\
                        \binom{1}{0} g_1(Sz) & \binom{1}{1}g_0(Sz) \\
                        \vdots & \vdots & \ddots \\
                        \binom{r}{0}g_r(Sz) & \binom{r}{1}g_{r-1}(Sz) & \cdots & \binom{r}{r}g_{0}(Sz)  
                    \end{pmatrix}
                    \begin{pmatrix}
                        1 \\
                        Sz \\
                        \vdots\\
                        (Sz)^r
                    \end{pmatrix}
            \end{eqnarray*}
        holds. This leads us to consider only the column vector
            \begin{eqnarray*}
                \mathbf{F}_U(Sz) = 
                \begin{pmatrix}
                    g_0(Sz) \\
                    \sum_{k=0}^{1} \binom{1}{k}g_{1-k}(Sz)(Sz)^k \\
                    \vdots \\
                    \sum_{k=0}^{r} \binom{r}{k}g_{r-k}(Sz)(Sz)^k
                \end{pmatrix}
            \end{eqnarray*}
        and anticipate that we will use Corollary \ref{cor:convolvUnderS}. Indeed, applying Corollary \ref{cor:convolvUnderS} we see \eqref{eq:FUnderS} holds as we wanted. 
    \end{proof}

\section{Summary}

We have introduced a vector function $ \mathbf{F}_U(z) $ dependent on an arbitrary quasiautomorphic form $ U_{\mathfrak{t}_\mu, w, r}(z) $ (weight $2 \leq w = 2n $, depth $ 0 \leq r \leq w/2 $, Hecke triangle group $ \mathfrak{t}_\mu $). This function we call a Hecke vector-form of $U$. Then we proved that that the functional equations
    \begin{eqnarray*}
        \mathbf{F}_U(Tz) &=& e^{\varpi_\mu A_r}\mathbf{F}_U(z) \\
        \frac{\mathbf{F}_U(Sz)}{z^{w-r}} &=& d_r^{\mathbf{y}}(a_i) \mathbf{F}_U(z)
    \end{eqnarray*}
hold (where $ a_i := (-1)^i $) modulo the generators of $ \mathfrak{t}_\mu $, namely $ Tz = z+ \varpi_\mu:= z+2\cos({\pi}/{\mu}) $ and $ Sz := -1/z $ .
\printbibliography

\end{document}